\documentclass[reqno]{amsart}

\usepackage{amssymb,latexsym}
\usepackage{amsfonts}
\usepackage{dsfont}
\usepackage[pdftex]{hyperref}
\usepackage{color}
\usepackage{graphicx}
\usepackage[all]{xy}

\theoremstyle{plain}
\newtheorem{thm}{Theorem}[section]
\newtheorem{lem}[thm]{Lemma}

\newtheorem{THM}{Theorem}

\newtheorem{rem}[thm]{Remark}
\newtheorem{defn}[thm]{Definition}

\numberwithin{equation}{thm}

\newcommand{\Z}{\mathbb Z}

\newcommand{\C}{\mathbb C}
\newcommand{\fg}{\mathfrak{g}}

\newcommand{\fb}{\mathfrak{b}}

\newcommand{\fm}{\mathfrak{m}}
\newcommand{\ft}{\mathfrak{t}}

\newcommand{\fu}{\mathfrak{u}}
\newcommand{\fq}{\mathfrak{q}}

\newcommand{\B}{\mathcal{B}}

\begin{document}

\title[The Connectedness of Hessenberg Varieties]{The Connectedness of Hessenberg Varieties}
\author{Martha Precup}
\address{Department of Mathematics\\Baylor University\\One Bear Place \#97328\\Waco, TX 76798}
\email{Martha\_Precup@baylor.edu}

\keywords{Hessenberg varieties, affine paving, rationally connected}
\subjclass[2010]{Primary 14L35, 14M15}

%
%
%

\begin{abstract}
In this paper we consider certain closed subvarieties of the flag variety, known as Hessenberg varieties.  We give a connectedness criterion for semisimple Hessenberg varieties generalizing a criterion given by Anderson and Tymoczko.  We show that nilpotent Hessenberg varieties are rationally connected.
\end{abstract}
\maketitle
\section{Introduction}

In this paper we explore the connectedness properties of semisimple and nilpotent Hessenberg varieties.  Hessenberg varieties are a family of subvarieties of the flag variety first introduced in \cite{dMPS}.  We give a criterion for semisimple Hessenberg varieties to be connected and prove that nilpotent Hessenberg varieties are rationally connected.

Let $G$ be a linear, reductive algebraic group over $\C$, $B$ a Borel subgroup, and let $\fg$, $\fb$ denote their respective Lie algebras.  A Hessenberg space $H$ is a linear subspace of $\fg$ that contains $\fb$ and is closed under the Lie bracket with $\fb$.  Fix an element $X\in \fg$ and a Hessenberg space $H$.  The Hessenberg variety, $\B(X,H)$, is the subvariety of the flag variety $G/B=\B$ consisting of all Borel subalgebras $g\cdot \fb$ such that $g^{-1}\cdot X\in H$ where $g\cdot X$ denotes the adjoint action $Ad(g)(X)$.

When $S\in \fg$ is semisimple, we say $\B(S,H)$ is semisimple.   In \cite{P} it is shown that all such Hessenberg varieties are paved by affines.  Applying this result, we have a method of computing the Betti numbers of $\B(S,H)$.  As a consequence, we are able to show the following, which is Theorem \ref{connS} below.

\begin{THM}  If $S\in \fg$ is a semisimple element which is not in the center of $\fg$, then $\B(S,H)$ is connected if and only if the Hessenberg space $H$ contains all root space vectors corresponding to the negative simple roots.
\end{THM}

The proof depends entirely on combinatorial properties of the root system associated to $\fb\subset \fg$.  This criterion was previously only known when $S\in \fg$ is a regular semisimple element (see \cite{AT}, Appendix A and \cite{dMPS}, Corollary 9(i)).   

When $N\in\fg$ is nilpotent, we say that $\B(N,H)$ is nilpotent.  When $N$ is regular in a Levi, $\B(N,H)$ is paved by affines (see \cite{T}, \cite{T2}, \cite{P}).  However, instead of using an affine paving to compute Betti numbers, we construct a sequence of rational curves connecting any two points in $\B(N,H)$, yielding the following for any choice of nilpotent $N\in \fg$.

\begin{THM}  The Hessenberg variety $\B(N,H)$, corresponding to any nilpotent element $N\in \fg$ and any Hessenberg space $H$, is rationally connected.
\end{THM}

The second section of this paper covers the background information and facts used in the following sections.  The third and fourth address the semisimple and nilpotent cases respectively.


\section{Preliminaries}
Let $G$ and $B$ be as above.  Let $T\subset B$ be a maximal torus with Lie algebra $\ft$ and denote by $W$ the Weyl group associated to $T$.  Fix a representative $\dot{w}\in N_G(T)$ for each Weyl group element $w\in W= N_G(T)/T$.  Let $\Phi^+$, $\Phi^-$ and $\Delta$ denote the positive, negative, and simple roots associated to the previous data.  Denote the root space corresponding to $\gamma\in\Phi$ by $\fg_{\gamma}$ and fix a generating root vector $E_{\gamma}\in \fg_{\gamma}$.  Write $U$ for the maximal unipotent subgroup of $B$, $U^-$ for the maximal unipotent subgroup of the opposite Borel, and $\fu$ and $\fu^-$ for their respective Lie algebras.  Denote the $1$-dimensional unipotent subgroup of $G$ corresponding to $\gamma\in \Phi$ by $U_{\gamma}$, i.e., $U_{\gamma}=exp(tE_{\gamma})$ for $t\in \C$. 

The projective variety $G/B$ is called the {\it flag variety} and is naturally identified with the set $\B$ of all Borel subalgebras of $\fg$.  The Hessenberg variety is a closed subvariety of the flag variety.  We define it here.

\begin{defn}  A subspace $H\subseteq \fg$ is a {\it Hessenberg space} with respect to $\fb$ if $\fb\subset H$ and $H$ is a $\fb$-submodule of $\fg$, i.e., $[\fb,H]\subseteq H$.
\end{defn}

Denote by $\Phi_H\subseteq \Phi$ the subset of roots such that $H=\ft \oplus \bigoplus_{\gamma \in\Phi_H} \fg_{\gamma}$.  Then the condition that $H$ be a Hessenberg space is equivalent to requiring that $\Phi^+\subseteq \Phi_H$ and $\Phi_H$ is closed with respect to addition with roots from $\Phi^+$.  Let $\Phi_H^-:=\Phi_H\cap \Phi^-$.

Let $X\in \fg$ and $H$ be a fixed Hessenberg space.  Set
		\[
				G(X,H)=\{ g\in G : g^{-1}\cdot X\in H \}
		\]
where $g\cdot X$ denotes $Ad(g)(X)$.  Then $G(X,H)$ is a closed subvariety of $G$ which is invariant under right multiplication by $B$.  

\begin{defn}  Let
		\[
				\B(X,H)=\{ g\cdot \fb \in \B : g\in G(X,H)    \}
		\]
denote the image of $G(X,H)$ in the flag variety $\B$.  This is the Hessenberg variety associated to $X$ and $H$.  If $X$ is semisimple we say $\B(X,H)$ is a semisimple Hessenberg variety and if $X$ is nilpotent we say $\B(X,H)$ is a nilpotent Hessenberg variety.  
\end{defn}

There are two extremal cases to be considered as examples.  First, when $H=\fb$, 
		\begin{eqnarray*}
				\B(X,\fb)=\{ g\cdot \fb : g^{-1}\cdot X\in \fb \}
					    =\{ g\cdot \fb : X\in g\cdot \fb \}
					     =:\B^X
		\end{eqnarray*}
is the variety of Borel subalgebras containing $X$, called the {\it Springer variety} or {\it Springer fiber}.  In the other extreme, when $H=\fg$, $\B(X,\fg)=\B$, the full flag variety. 

In what follows we give a criterion for semisimple Hessenberg varieties to be connected and prove that all nilpotent Hessenberg varieties are rationally connected.  The first makes use of the fact that all semisimple Hessenberg varieties are paved by affines, yielding combinatorial methods of calculating their Betti numbers.  The second result is constructive, providing insight into the structure of nilpotent Hessenberg varieties.

%

\section{Semisimple Hessenberg Varieties} \label{ss}


To give a criterion for semisimple Hessenberg varieties to be connected we use the fact that they are paved by affines. 
 
\begin{defn}  A paving of an algebraic variety $Y$ is a filtration by closed subvarieties
		\[
				Y_0\subset Y_1 \subset \cdots \subset Y_i \subset \cdots \subset Y_d=Y.
		\]	
A paving is affine if every $Y_i-Y_{i-1}$ is a finite disjoint union of affine spaces.
\end{defn}

An affine  paving computes the Betti numbers of an algebraic variety as follows.

\begin{lem}\label{betti} (\cite{F}, 19.1.11)  Let $Y$ be an algebraic variety with an affine paving, $Y_0\subset Y_1 \subset \cdots \subset Y_i \subset \cdots \subset Y_d=Y$.  Then the nonzero compactly supported cohomology groups of $Y$ are given by $H_c^{2k}(Y)= \Z^{n_k}$ where $n_k$ denotes the number of affine components of dimension $k$.
\end{lem}

It is well known that the flag variety $\B$ is paved by affines.  Apply the Bruhat decomposition to write $\B=\bigsqcup_{w\in W} X_w$ where $X_w=B\dot{w}B/B$ is the Schubert cell indexed by $w\in W$.  The paving of $\B$ is given by 
		\[
				\B_i=\bigsqcup_{l(w)=i} \overline{X_w}
		\]  
where $l$ indicates the length function on $W$.  Since $\overline{X_w}=\bigsqcup_{y\leq w} X_y$, where $\leq$ denotes the Bruhat order and $X_w\cong \C^{l(w)}$, this paving is affine.

Let $S\in \fg$ be a semisimple element.  To show that the semisimple Hessenberg variety $\B(S,H)$ is paved by affines, we consider the intersections $\B_i\cap \B(S,H)$.  Since the $\B_i$ pave $\B$, these intersections pave $\B(S,H)$ and we have only to show that 
		\[
				\B_i\cap \B(S,H)-\B_{i-1}\cap\B(S,H)=\bigsqcup_{w\in W;\,l(w)=i} X_w\cap \B(S,H) 
		\]
is a finite disjoint union of affine spaces.  In other words, it is enough show that each $X_w\cap \B(S,H)$ is isomorphic to affine space for all $w\in W$.  This result can be found in Theorem 5.4 and Remark 5.6 of \cite{P}.  To calculate the Betti numbers of $\B(S,H)$ using Lemma \ref{betti} we need an understanding of $\dim(X_w\cap \B(S,H))$ for all $w\in W$.

Let $M=Z_G(S)$, the centralizer of $S$ in $G$.  We fix a basis so that $S\in \ft$ and $M$ is a standard Levi subgroup of $G$.  Denote its Lie algebra by $\fm$, and let $Q=MU_Q$ denote the corresponding standard parabolic subgroup, with Lie algebra $\fq=\fm\oplus \fu_Q$.  Recall that given a standard Levi subgroup $M$ of $G$, we write $W_M$ to denote the subgroup of the Weyl group associated to $M$.  Let
		\[
				W^M=\{ v\in W : \Phi_v\subseteq \Phi(\fu_Q) \}
		\]  
where $\Phi_w=\{ \gamma\in \Phi^+ : w^{-1}(\gamma)\in \Phi^- \}$ for $w\in W$.  The elements of $W^M$ form a set of minimal coset representatives for $W_M \backslash W$ since each $w\in W$ can be written uniquely as $w=yv$ with $y\in W_M$, $v\in W^M$, and so that $l(w)=l(y)+l(v)$.  Consequently, $\Phi_w=\Phi_y \sqcup y(\Phi_v)$.

\begin{lem}\label{dimen} (\cite{P}, Corollary 5.8(1))  Fix a Hessenberg space $H$ with respect to $\fb$.  Then for all $w=yv$, where $y\in W_M$ and $v\in W^M$, $X_w\cap \B(S,H)$ is nonempty and
		\[
				\dim(X_w\cap \B(S,H))=| \Phi_y |+| \Phi_v \cap v(\Phi_H^-)|.
		\]
\end{lem}

Using Lemma \ref{betti}, we can now give a criteria for $\B(S,H)$ to be connected.  

\begin{thm}\label{connS}  If $S\in \fg$ is a semisimple element which is not in the center of $\fg$ and $H\subset \fg$ is a Hessenberg space with respect to $\fb$, then $\B(S,H)$ is connected if and only if $-\Delta\subseteq \Phi_H^-$.
\end{thm}
\begin{proof}  Recall that the dimension of the 0-cohomology group of any topological space is equal to the number of connected components of that space.  Therefore, by Lemma \ref{betti}, it is enough to show that for all $w\in W$ not equal to the identity, $\dim (X_w\cap \B(S,H))>0$ if and only if $-\Delta \subseteq \Phi_H^-$.  To prove this, we'll use the fact that $w^{-1}(\Phi_w)=-\Phi_{w^{-1}}$.

($\Rightarrow$) Suppose there exists $-\alpha\in -\Delta$ such that $-\alpha\notin \Phi_H^-$.  Our goal is to find some $v\in W^M - \{ e \}$ such that
		\begin{eqnarray}\label{star}
				v^{-1}(\Phi_v)\cap \Phi_H^-=\emptyset
		\end{eqnarray}
since then $\dim(X_v\cap \B(S,H))=|\Phi_v\cap v(\Phi_H^-)|=|v^{-1}(\Phi_v) \cap \Phi_H^-|=0$ by Lemma \ref{dimen}, implying that $\B(S,H)$ is not connected.  There are two cases we must consider.

First, if $\alpha\in \Phi(\fu_{Q})$ then we can take $v=s_{\alpha}\in W^M$ where $s_{\alpha}$ is the simple reflection corresponding to $\alpha\in\Delta$.  In this case, $s_{\alpha}^{-1}(\Phi_{s_\alpha})=\{-\alpha\}$ so equation (\ref{star}) is satisfied.

Second, say $\alpha\in \Phi(\fm)$.  We identify the element $v\in W^M$ as follows.  Let
		\[
				\Phi_{\geq \alpha}=\{ \gamma\in \Phi^+:\gamma\geq \alpha \}
		\]
where $\geq$ is the partial ordering on positive roots defined in \cite{H}, \S 10.1.  Note that $-\Phi_{\geq\alpha}\cap \Phi_H^-=\emptyset$.  Indeed, if $\gamma\in \Phi_{\geq \alpha}$, write $\gamma=\alpha_1+\cdots + \alpha_k$ with $\alpha_i\in \Delta$ (not necessarily distinct) so that each partial sum $\alpha_1+\cdots+\alpha_i$ is a positive root (see \cite{H}, \S10.2 Lemma A).  Since $\gamma\in \Phi_{\geq \alpha}$, there exists some $j\in \{1,...,k\}$ so that $\alpha_j=\alpha$.  Now, if $-\gamma\in \Phi^-_H$ then $-\gamma+\alpha_k =-\alpha_1-\cdots-\alpha_{k-1}\in \Phi_H^-$ by definition of the Hessenberg space.  Continuing in this fashion, we may assume $-\alpha_1-\cdots -\alpha_j\in \Phi_H^-$.  But this forces
		\[
				-\alpha=-\alpha_j=(-\alpha_1-\cdots-\alpha_j)+(\alpha_1+\cdots +\alpha_{j-1})\in \Phi_H^-
		\]
since $\alpha_1+\cdots + \alpha_{j-1}\in \Phi^+$, which is a contradiction since by assumption $-\alpha \notin \Phi_H^-$.  In addition, since $S$ is not an element of the center of $\fg$, $\Phi_M \subsetneqq \Phi$.  In particular, we may assume $\theta$, the highest root in $\Phi^+$, is an element of $\Phi(\fu_Q)$.  Since $\theta \geq \alpha$, we conclude that $\Phi(\fu_Q)\cap \Phi_{\geq \alpha}\neq \emptyset$.  

Note also that $\Phi_{\geq \alpha}$ is closed with respect to addition, as is its compliment $\Phi_{\geq\alpha}^c:=\Phi^+-\Phi_{\geq \alpha}$.  This is clear since $\Phi_{\geq\alpha}$ are the roots associated to the nilpotent ideal of $\fb$ in the Levi decomposition of the standard parabolic associated to the subset $\Delta-\{\alpha\}$ of simple roots.


Since $\Phi_{\geq\alpha}$ and $\Phi_{\geq\alpha}^c$ are both closed with respect to addition, by \cite{K}, Proposition 5.10, there exists some $w\in W$ so that $\Phi_w=\Phi_{\geq \alpha}$.  Consider $w^{-1}=yv\in W$ where $y\in W_M$ and $v\in W^M$.  We must have $w^{-1}\notin W_M$ since if $w\in W_M$, then $\emptyset=\Phi(\fu_{Q})\cap \Phi_w=\Phi(\fu_Q)\cap \Phi_{\geq\alpha}$, which is a contradiction.  Therefore $v\neq e$.  We claim that this element $v\in W^M$ satisfies the property in equation (\ref{star}).  Write $v=y^{-1}w^{-1}$.  We see that
		\[
				\Phi_{w^{-1}}=\Phi_y \sqcup y(\Phi_v) \Rightarrow y(\Phi_v)=\Phi_{w^{-1}}-\Phi_y\Rightarrow \Phi_v									=y^{-1}(\Phi_{w^{-1}}-\Phi_y),
		\]
and therefore,
		\[
				v^{-1}(\Phi_v)=v^{-1}y^{-1}(\Phi_{w^{-1}}-\Phi_y)\subseteq v^{-1}y^{-1}(\Phi_{w^{-1}})=w(\Phi_{w^{-1}})								=-\Phi_w=-\Phi_{\geq\alpha}.
		\]
Since $-\Phi_{\geq \alpha}\cap \Phi_H^-=\emptyset$, we get that $v^{-1}(\Phi_v)\cap \Phi_H^-=\emptyset$, as desired.  We have now shown that if $-\Delta \nsubseteq \Phi_H^-$, then $\B(S,H)$ is not connected.

($\Leftarrow$)  Alternatively, say $\B(S,H)$ is not connected.  Then there exists some $w\in W$ so that $\dim (X_w\cap \B(S,H))=0$ and $w$ is not the identity.  By Lemma \ref{dimen} this can be the case only if $w\in W^M$ and
		\[
				w(\Phi_H^-)\cap \Phi_w=\emptyset\Rightarrow \Phi_H^-\cap w^{-1}(\Phi_w)=\emptyset											\Rightarrow \Phi_H^-\cap -\Phi_{w^{-1}}=\emptyset.
		\]
But since $w^{-1}\neq e$, there must be some simple root $\alpha\in \Phi_{w^{-1}}$.  Since $-\alpha\in -\Phi_{w^{-1}}$, we get $-\alpha\notin \Phi_H^-$.  Therefore, if $\B(S,H)$ is not connected, $-\Delta\nsubseteq \Phi_H^-$, and we have proven the Theorem.
\end{proof}

\begin{rem}  If the semisimple element $S\in \fg$ is an element of the center of $\fg$ then this criterion no longer applies.  In this case, $\B(S,H)$ is the full flag variety and is therefore connected.
\end{rem}

Theorem \ref{connS} extends a similar result from the appendix of \cite{AT} which gives the same criterion for $\B(S,H)$ to be connected when $S\in \fg$ is regular semisimple.  It also extends Corollary 9(i) from \cite{dMPS} which gives an equivalent criterion for $\B(S,H)$ to be connected when $S$ is regular semisimple.


\section{Nilpotent Hessenberg Varieties} \label{Nilp}


In this section we show that any nilpotent Hessenberg variety is rationally connected.  Our methods are similar to those of Tits proving that any Springer variety is rationally connected (see \cite{J}, \S10).  Namely, both arguments rely on the selection of a rational curve connecting any two points.  Our methods differ from those of Tits in that we must utilize non-simple root subgroups to do this for the Hessenberg variety.

One can also show that nilpotent Hessenberg varieties in certain cases are connected by calculating Betti numbers via an affine paving as we did above for semisimple varieties.  For example, when $G=GL_n(\C)$, Tymoczko showed in \cite{T} that the intersections $X_w\cap \B(N,H)$ for $N\in \fg$ nilpotent are affine with dimensions given by certain fillings of the Young diagram associated to $N$.  In \cite{I}, Iveson considers these fillings in greater detail and proves that $\dim(X_w\cap \B(N,H))>0$ for all $w\in W-\{e\}$.  In contrast, the methods below show each $\B(N,H)$ is rationally connected for all types and does not depend on an affine paving or choice of element $N\in \fg$.

Fix a a representative of $N$ so that $N=\sum_{\gamma\in \Phi_N} E_{\gamma} \in \fu$ for some subset, $\Phi_N$, of positive roots.  For each $w\in W$ consider the following subsets of positive roots:
		\begin{eqnarray*}
				\Phi_w&=& \{ \gamma\in \Phi^+ : w^{-1}(\gamma)\in \Phi^- \}\\
				\Phi_w^c&=& \Phi^+ - \Phi^w= \{ \gamma\in \Phi^+ : w^{-1}(\gamma)\in \Phi^+ \}.
		\end{eqnarray*}
We define corresponding subalgebras of $\fu$:
		\[
				\fu_w:=\bigoplus_{\gamma\in \Phi_w}\fg_{\gamma} 																	\mbox{  and  }  \fu_w^c:= \bigoplus_{\gamma\in \Phi_w^c} \fg_{\gamma}
		\]
so $\fu=\fu_w\oplus \fu_w^c$.

Let $H$ be a fixed Hessenberg space with respect to $\fb$ and recall that $\Phi_H=\Phi^+\sqcup \Phi_H^-$.  Note that:
		\begin{eqnarray*}
				w(\Phi_H)\cap \Phi^+ &=& (w(\Phi^+) \cap \Phi^+) \sqcup (w(\Phi_H^-) \cap \Phi^+)\\
							&=& \Phi_w^c \sqcup (w(\Phi_H^-) \cap \Phi^+ ). 
		\end{eqnarray*} 
In particular, we can state the following.

\begin{rem}\label{Remark 1}  For all $w\in W$, $\fu_w^c\subseteq \dot{w}\cdot H$.
\end{rem}

Now, let $\Phi_w^{m}=\{ \gamma\in \Phi_w : \alpha \leq \gamma \; \forall \alpha\in \Phi_w \}\subseteq \Phi_w$ where $\leq$ is the partial order on positive roots as in the proof of Theorem \ref{connS}.  

\begin{rem}\label{Remark 2}  Note that $\Phi_w^{m}\neq \emptyset$ for all $w\in W-\{e\}$ and by definition, if $\gamma'\in \Phi^+$ such that $\gamma < \gamma'$ for some $\gamma\in \Phi_w^{m}$ then $\gamma' \in \Phi_w^c$.
\end{rem}

\begin{lem}\label{LEM}  If $\dot{w}\cdot \fb\in \B(N,H)$, then $U_{\gamma}\dot{w}\cdot \fb\subset \B(N,H)$ for all $\gamma\in \Phi_w^m$.  
\end{lem}
\begin{proof}  Indeed, let $u\in U_{\gamma}$ for $\gamma\in \Phi_w^m$.  The condition that $U_{\gamma}\dot{w}\cdot \fb\subset \B(N,H)$ is equivalent to showing $u^{-1}\cdot N\in \dot{w}\cdot H$ for all such $u$.  Write $u=exp(tE_{\gamma})$ with $t\in \C$.  Then
		\begin{eqnarray*}
				u^{-1}\cdot N &=& exp(-tE_{\gamma})\cdot N
						 = exp(ad_{-tE_{\gamma}})(N)\\
						&=& \sum_{i=0}^{\infty} \frac{ad_{-t E_{\gamma}}^i (N)}{i!}
						= N +\sum_{\gamma'\in \Phi(\gamma,N)} d_{\gamma'}E_{\gamma'}
		\end{eqnarray*}
where $\Phi(\gamma,N)=\{ \gamma'\in \Phi^+ : \gamma'=c\gamma+\alpha,\, c\in\{1, 2, 3\},\, \alpha\in\Phi_N \}$ and $d_{\gamma'}\in \C$.  Note that this is a finite expression, since $E_{\gamma}$ is a nilpotent element of $\fg$.  Also, all $\gamma'\in \Phi(\gamma,N)$ have the property that $\gamma < \gamma'$.  In particular, by Remark \ref{Remark 2}, $\gamma'\in \Phi_w^c$ so $E_{\gamma'}\in \fu_w^c\subseteq \dot{w}\cdot H$ by Remark \ref{Remark 1}.  Therefore, since $N\in \dot{w}\cdot H$,
		\[
				u^{-1}\cdot N=N +\sum_{\gamma'\in \Phi(\gamma,N)} d_{\gamma'}E_{\gamma'}\in \dot{w}\cdot H,
		\]
and $U_{\gamma}\dot{w}\cdot \fb\subset \B(N,H)$.
\end{proof}

Since $\Phi_w^m\neq \emptyset$, we have shown that for any nilpotent element $N\in \fu$, if $\dot{w}\cdot \fb\in \B(N,H)$ then there exists some $\gamma\in \Phi_{w}$ so that $U_{\gamma}\dot{w}\cdot \fb\subset \B(N,H)$.

\begin{thm} \label{connN}  The Hessenberg variety $\B(N,H)$ corresponding to any nilpotent element $N\in \fg$ and any Hessenberg space $H$ is rationally connected.
\end{thm}
\begin{proof}  Suppose $\B(N,H)\cap X_w\neq \emptyset$, i.e., there exists some $u_0\in U$ such that $u_0\dot{w}\cdot \fb\in \B(N,H)$.  Let $N_0:=u_0^{-1}\cdot N \in \fu$.  Then:
		\[
				u_0\dot{w}\cdot \fb\in \B(N,H) \Leftrightarrow u_0^{-1}\cdot N\in \dot{w}\cdot H 													\Leftrightarrow N_0\in \dot{w}\cdot H \Leftrightarrow \dot{w}\cdot \fb\in \B(N_0,H). 
		\]
By Lemma \ref{LEM}, since $N_0\in \fu$ and $\dot{w}\cdot \fb\in \B(N_0,H)$ there exists $\gamma\in \Phi_w$ so that $U_{\gamma}\dot{w}\cdot \fb \subset \B(N_0,H)$.  But this is the case if and only if for all $u\in U_{\gamma}$,
		\[
				u^{-1}\cdot N_0\in \dot{w}\cdot H \Leftrightarrow u^{-1}u_0^{-1}\cdot N\in \dot{w}\cdot H
		\]
in other words, $u_0U_{\gamma}\dot{w}\cdot \fb\subset \B(N,H)$.  Since $\gamma\in \Phi_w$, $s_{\gamma}w<w$ and
		\[
				\lim_{t\to \infty} u_0U_{\gamma}\dot{w}\cdot \fb=\lim_{t\to \infty} u_0exp(tE_{\gamma})\dot{w}\cdot \fb								=u_0\dot{s}_{\gamma}\dot{w}\cdot \fb\in \B(N,H)
		\]
(see the proof of Theorem 2.11 in \cite{BGG}).  We have shown that $u_0\dot{w}\cdot \fb$ is connected to $u_0\dot{s}_{\gamma}\dot{w}\cdot \fb$ by the rational curve $\overline{u_0U_{\gamma}\dot{w}\cdot\fb}$ in $\B(N,H)$.  It follows by induction on $l(w)$ that if $u_0\dot{w}\cdot \fb\in \B(N,H)$, then it is connected by a sequence of rational curves in $\B(N,H)$ to $u_0\cdot \fb=\fb$.
\end{proof}
		

%

{\bf Acknowledgements.} The author would like to thank Sam Evens and Markus Hunziker for providing useful feedback on drafts of the above.

\end{document}